\documentclass[a4paper,10pt]{amsart}
\usepackage[utf8]{inputenc}
\usepackage{a4wide}
\usepackage{amsfonts,amssymb,amsthm,amsmath}
\usepackage{color,url,enumitem,hyperref}

\theoremstyle{plain}
\newtheorem{theorem}{Theorem}[section]
\newtheorem{thmintro}{Theorem}[section]
\newtheorem{lemma}[theorem]{Lemma}
\newtheorem{corollary}[theorem]{Corollary}

\newtheorem{conjintro}[thmintro]{Conjecture}
\newtheorem{problemintro}{Problem}

\theoremstyle{definition}
\newtheorem{definition}[theorem]{Definition}

\definecolor{darkblue}{rgb}{0,0,0.7} 
\newcommand{\darkblue}{\color{darkblue}} 
\newcommand{\defn}[1]{\emph{\darkblue #1}} 

\DeclareMathOperator{\join}{\ast}
\DeclareMathOperator{\susp}{\Sigma} 
\DeclareMathOperator{\lk}{lk} 
\DeclareMathOperator{\st}{st} 
\DeclareMathOperator{\del}{del} 
\DeclareMathOperator{\verti}{V} 
\DeclareMathOperator{\edge}{E} 

\title{Bounds for entries of $\gamma$-vectors of flag homology spheres}

\author[J.-P.~Labb\'e]{Jean-Philippe Labb\'e$^{1}$}
\address[J.-P. Labb\'e]{Einstein Institute of Mathematics, Hebrew University of Jerusalem, Jerusalem 91904, Israel}
\email{labbe@math.fu-berlin.de}
\urladdr{http://page.mi.fu-berlin.de/labbe/}
\thanks{$^{1}$With the support of a FQRNT post-doctoral fellowship and a post-doctoral ISF-805/11 grant}

\author[E.~Nevo]{Eran Nevo$^{2}$}
\address[E.~Nevo]{Einstein Institute of Mathematics, Hebrew University of Jerusalem, Jerusalem 91904, Israel}
\email{nevo@math.huji.ac.il}
\urladdr{http://math.huji.ac.il/~nevo/}
\thanks{$^{2}$Partially supported by Israel Science Foundation grants ISF-805/11 and ISF-1695/15}

\begin{document}

\begin{abstract}
We present some enumerative and structural results for flag homology spheres.
For a flag homology sphere $\Delta$, we show that its $\gamma$-vector $\gamma^\Delta=(1,\gamma_1,\gamma_2,\ldots)$ satisfies:
\begin{align*}
  \gamma_j=0,\text{ for all } j>\gamma_1, \quad \gamma_2\leq\binom{\gamma_1}{2}, \quad \gamma_{\gamma_1}\in\{0,1\}, \quad \text{ and }\gamma_{\gamma_1-1}\in\{0,1,2,\gamma_1\},
\end{align*}
supporting a conjecture of Nevo and Petersen.
Further we characterize the possible structures for~$\Delta$ in extremal cases.
As an application, the techniques used produce infinitely many $f$-vectors of flag balanced simplicial complexes that are not $\gamma$-vectors of flag homology spheres (of any dimension); these are the first examples of this kind.

In addition, we prove a flag analog of Perles' 1970 theorem on $k$-skeleta of polytopes with ``few'' vertices, specifically:
the number of combinatorial types of $k$-skeleta of flag homology spheres with $\gamma_1\leq b$, of any given dimension, is bounded independently of the dimension.
\end{abstract}

\maketitle

\section{Introduction}
Facial enumeration of polytopes and simplicial spheres are of great interest since antiquity, and research on this topic revealed fascinating mathematics, notably in the celebrated $g$-theorem which characterizes the face numbers of simplicial $d$-polytopes \cite{billera_sufficiency_1980,stanley_number_1980,stanley_combinatorics_1996}.

In relation to the Charney--Davis conjecture, a subfamily of great importance is that of \emph{flag} simplicial polytopes or more generally \emph{flag} homology spheres.
These objects have the property that their faces are exactly the cliques of their graphs, equivalently all their minimal non-faces have two elements.
In the flag case the right analog of the $g$-vector (from the simplicial case) seems to be Gal's $\gamma$-vector.
Gal conjectured that flag homology spheres have nonnegative $\gamma$-vectors \cite{gal_real_2005} which, by the same token, would validate the Charney--Davis conjecture.
Later Nevo and Petersen strengthened Gal's conjecture by proposing the following combinatorial interpretation of the $\gamma$-vector.

\begin{conjintro}\cite{nevo_gamma_2011}\label{conj:Nevo-Petersen}
  Let $\Delta$ be a flag homology sphere. The $\gamma$-vector of $\Delta$ is the $f$-vector of a flag simplicial complex.
\end{conjintro}

Partial results toward this conjecture include \cite{gal_real_2005, nevo_gamma_2011, Nevo-Petersen-Tenner_2011, aisbett_gamma_2012, volodin_cubical_2010, Nevo-Murai_gamma_2012, zheng_flag_2015, adamaszek_upper_2016}.
The recent upper bound results of Adamaszek--Hladk\'{y} and of Zheng \cite{adamaszek_upper_2016,zheng_flag_2015} on the entries of the $\gamma$-vector require $\gamma_1>>\dim(\Delta)$, or $\dim(\Delta)\in\{3,5\}$ respectively.
Conjecture~\ref{conj:Nevo-Petersen} implies the following upper and lower bounds for the entries $\gamma_i$ of the $\gamma$-vector  in terms of $\gamma_1$ alone.

\begin{conjintro}\label{conj:UBC}
  Let $\Delta$ be a flag homology sphere with $\gamma_1=\ell$. For $i\geq 2$, the entries of the $\gamma$-vector of $\Delta$ satisfy
  \[
    0 \leq \gamma_i\leq \binom{\ell}{i}.
  \]
\end{conjintro}

Our first main result is a proof of parts of this conjecture.

\begin{thmintro}\label{thm:intro_c}
  Let $\Delta$ be a flag homology sphere with $\gamma_1=\ell$. The entries of the $\gamma$-vector of $\Delta$ satisfy the following properties:
  \begin{enumerate}[label=\roman{enumi}),ref=\ref{thm:intro_c}~\roman{enumi})]
   \item $\gamma_2\leq \binom{\ell}{2}$, \label{thm:intro_ci}
   \item $\gamma_j=0$ for all $j>\ell$, \label{thm:intro_cii}
   \item $\gamma_\ell\in\{0,1\}$, \label{thm:intro_ciii}
   \item $\gamma_{\ell-1}\in\{0,1,2,\ell\}$. \label{thm:intro_civ}
  \end{enumerate}
\end{thmintro}

Further, we characterize the structure of extremal cases in the above theorem, omitted in the introduction for brevity.
It is important to notice that the property in Theorem~\ref{thm:intro_cii} contrasts with $g$-vectors of simplicial polytopes: fixing any $b\geq 1$ there exists, for any integer $i>1$, a polytope $P_i$ with $g_1(P_i)\leq b$ and $g_i(P_i)>0$.

A basic ingredient in the proof of Theorem~\ref{thm:intro_cii} to \ref{thm:intro_civ} is to show that if the dimension of $\Delta$ is large enough compared with $\gamma_1$ then $\Delta$ is a suspension over a lower dimensional flag homology sphere. Similarly to Perles' \cite[Theorem 1.1]{kalai_aspects_1994}, we obtain the following analog as a corollary of the suspension result.

\begin{thmintro}\label{thm:intro_d}
  Let $F_{b,k}(d)$ denote the number of combinatorial types of $k$-skeleta of flag homology $(d-1)$-spheres with at most $2d+b$ vertices.
  Given $b$ and $k$, the set $\{F_{b,k}(i)\}_{i\geq 1}$ is bounded, hence finite.
\end{thmintro}

Does the converse of Conjecture~\ref{conj:Nevo-Petersen} hold? Namely, is it true that for any flag simplicial complex~$\Gamma$, there exists a flag homology sphere $\Delta$, of any dimension, such that $\gamma^\Delta$ equals the $f$-vector of $\Gamma$?
We provide the first counterexamples.

\begin{thmintro}\label{thm:intro_e}
  Let $k\geq 3$ and $\Gamma_k$ be the flag balanced simplicial complex consisting of the disjoint union of a $(k-1)$-simplex and an isolated vertex. The $f$-vector of $\Gamma_k$ is not the $\gamma$-vector of any flag homology sphere of any dimension.
\end{thmintro}

Currently, a guess characterization of $\gamma$-vectors of flag homology spheres or flag simplicial polytopes seems to lack.
In relation with the $g$-conjecture for homology spheres this poses the following problem.

\begin{problemintro}\label{prob:sphere/polytope}
Does there exist a vector which is the $\gamma$-vector of a flag homology sphere but not the $\gamma$-vector of a flag simplicial polytope?
\end{problemintro}

Answering ``No'' to this question but ``Yes'' for the corresponding problem on $f$-vectors may be possible, as in Problem~\ref{prob:sphere/polytope} we do not insist that the homology sphere and the polytope boundary have the same dimension. Given a flag homology sphere, the range of relevant dimensions for the flag polytopes is bounded though, due to the suspension result mentioned above; see Lemmas~\ref{lem:gamma_suspension} and~\ref{lem:bound_dim}.
We remark that in the non-simplicial case, and for the finer flag $f$-vector invariant, Brinkmann and Ziegler recently found a flag $f$-vector of a polyhedral $3$-sphere which is not the flag $f$-vector of any $4$-polytope \cite{brinkmann_small_2016}.

\medskip

{\bf Outline.} In Section~\ref{sec2:back} we set notation and give background facts on flag homology spheres.
In Section~\ref{sec3:gamma2} we prove the upper bound on $\gamma_2$ and determine the extremal cases (Theorems~\ref{thm:intro_ci}).
In Section~\ref{sec4:gamma_gamma1} we prove the upper bound on $\gamma_{\gamma_1}$, determine the extremal cases and prove a flag analog of Perles' theorem (Theorems~\ref{thm:intro_cii}, \ref{thm:intro_ciii} and~\ref{thm:intro_d}).
In Section~\ref{sec5:gamma_gamma1-1} we prove the upper bound on $\gamma_{\gamma_1-1}$, determine the extremal cases, and use it to give examples of $f$-vectors of flag complexes which are not the $\gamma$-vector of any flag homology sphere (Theorems~\ref{thm:intro_civ} and \ref{thm:intro_e}).

{\bf Acknowledgements.} The authors are thankful to the referees and Marc Masdeu for helpful suggestions and corrections to the manuscript.

\section{Background and definitions}~\label{sec2:back}
\subsection{Simplicial complexes}
Let $n\in\{1,2,\dots\}$, a \defn{simplicial complex} $\Delta$ is a collection of subsets, called \defn{faces}, of $[n]:=\{1,\dots, n\}$ closed under containment.
A face $f$ of cardinality $|f|=k$ is said to have \defn{dimension}~$k-1$.
Faces of dimension $0$ (respectively of dimension $1$) are called \defn{vertices} (resp. \defn{edges}).
We denote the set of vertices of $\Delta$ as~$\verti(\Delta)$ and the set of edges as $\edge(\Delta)$.
The degree $\deg(v)$ of a vertex $v$ of $\Delta$ is the number of edges of $\Delta$ containing $v$.

The \defn{empty complex} is the simplicial complex containing only the empty set as a face.
A subcomplex $\Delta'\subseteq\Delta$ is called \defn{induced} if for each face $f\in\Delta$ such that $f\subseteq \verti(\Delta')$ then $f\in \Delta'$.
The inclusion maximal faces are called \defn{facets}.
For a finite set $f$ we denote by $<f>$ the simplicial complex with the unique facet $f$.
Given a face $f\in\Delta$, the~\defn{star},~\defn{deletion}, and~\defn{link} of~$f$ in~$\Delta$ are the following subcomplexes of $\Delta$, where $\sqcup$ denotes the disjoint union:
\begin{align*}
\st_\Delta(f) & := \{f' \in \Delta : f\cup f' \in \Delta \}, \\
\del_\Delta(f) & := \{f'\in \Delta : f\not\subseteq f'\}, \\
\lk_\Delta(f) & := \{f' \in \Delta : f \sqcup f' \in \Delta\}.
\end{align*}
It is important to notice that the deletion of a face $f$ does not remove the faces properly contained in $f$.
We extend the definition of deletion to any subcomplex $\Gamma$ of $\Delta$, by deleting all vertices of $\Gamma$, namely $\del_\Delta(\Gamma):=\bigcap_{v\in \verti(\Gamma)} \del_\Delta(v) = \{f\in \Delta: f\subseteq [n]\setminus \verti(\Gamma)\}$.

A simplicial complex $\Delta$ is \defn{pure} if all facets of $\Delta$ have the same cardinality~$d$, in which case we say that $\Delta$ has \defn{dimension} $d-1$.
We say that a pure $(d-1)$-dimensional simplicial complex $\Delta$ is a \defn{pseudomanifold} if every face of size $(d-1)$ is contained in either 1 or 2 facets.
To each face of dimension $k-1$ of a simplicial complex, we can associate a geometric simplex of dimension~$k-~1$.
The \defn{geometric realization} $||\Delta||$ of $\Delta$ is the unique topological space, up to homeomorphism, obtained by gluing the geometric simplices corresponding to faces of $\Delta$ along their intersections.
If the geometric realization $||\Delta||$ is homeomorphic to a sphere, we say that $\Delta$ is a \defn{simplicial sphere}.
Moreover, given a pure $(d-1)$-dimensional simplicial complex $\Delta$ and a field $\mathbb{K}$, we say that $\Delta$ is a \defn{homology sphere} (over $\mathbb{K}$) if for all faces $f\in \Delta$ we have
\[
	\widetilde{H}_i(\lk_\Delta(f),\mathbb{K})=\begin{cases} 0 & \text{if } i<d-1-|f|, \\ \mathbb{K} & \text{if } i=d-1-|f|,\end{cases}
\]
where $\widetilde{H}_*(\Gamma,\mathbb{K})$ denotes the reduced singular homology of $||\Gamma||$ with coefficients in $\mathbb{K}$.
A pure $(d-1)$-dimensional simplicial complex $B$ is a \defn{homology ball} over the field $\mathbb{K}$ if (i) for all faces $f\in \Delta$ the link $\lk_\Delta(f)$ is either a $(d-1-|f|)$-dimensional homology sphere over $\mathbb{K}$ or is homologically $\mathbb{K}$-acyclic, and (ii) the faces of $B$ with acyclic link form a subcomplex which is a $(d-2)$-dimensional homology sphere over $\mathbb{K}$.
Given a homology sphere $\Delta$, an induced codimension-1 homology sphere $\Gamma\subseteq\Delta$ is called an \defn{equator} of~$\Delta$.
By Jordan--Alexander theorem, deleting an equator $\Gamma$ from a flag homology sphere $\Delta$, we obtain two disjoint acyclic complexes by $\del_\Delta(\Gamma)^+$ and $\del_\Delta(\Gamma)^-$.
Then, deleting each part from $\Delta$ respectively, we get two homology balls
$\Delta_\Gamma^+$ and $\Delta_\Gamma^-$, called \defn{hemispheres},
intersecting in $\Gamma$, and $\Delta$ decomposes as $\Delta=\Delta_\Gamma^+ \cup_\Gamma \Delta_\Gamma^-$.

Here are some essential operations on simplicial spheres.
The~\defn{join} of two simplicial complexes~$\Delta_1$ and $\Delta_2$ is the simplicial complex
\[
\Delta_1\join \Delta_2:=\{f_1\sqcup f_2 : f_1\in \Delta_1,f_2\in \Delta_2\}.
\]
We denote the $k$-fold join of a fixed simplicial complex $\Delta$ by $\join^k \Delta$, i.e. $\join^0\Delta$ is the empty complex and $\join^1\Delta=\Delta$.
The \defn{suspension} $\susp \Delta$ of a simplicial complex $\Delta$ is the join of $\Delta$ with a $0$-dimensional sphere $S^0$, i.e. the simplicial complex with two disjoint vertices,
\[
 \susp\Delta := S^0 \join \Delta.
\]
We denote the $k$-fold suspension $\underbrace{\Sigma\cdots\Sigma}_{k \text{ times}}\Delta$ by $\susp^k\Delta$.
The $k$-fold suspension of the empty complex is the $(k-1)$-dimensional simplicial complex isomorphic to the boundary complex of the so-called cross-polytope of dimension $k$, also called \defn{octahedral sphere}.
Given an edge $e=\{a,b\}$ of a simplicial complex $\Delta$, define the function $\kappa_e:[n]\rightarrow [n]$ sending $i\mapsto i$ if $i\notin \{a,b\}$ and $i\mapsto a$ otherwise. The \defn{contraction} of $e$ in $\Delta$ is the simplicial complex
\[
 \Delta^{\downarrow e} := \{\kappa_e(f): f\in \Delta\}.
\]

A \defn{non-face} of a simplicial complex $\Delta$ is a set of vertices of~$\Delta$ which is not a face of $\Delta$.
A \defn{missing face} (also called minimal non-face) of $\Delta$ is a non-face $g$ such that each proper subset of $g$ is a face of $\Delta$.
In this article, we always assume that all vertices are faces, that is, missing faces of simplicial complexes have dimension at least 1.
When all missing faces of $\Delta$ have dimension~$1$, the simplicial complex $\Delta$ is called \defn{flag}.
We denote by $C_k$ the $1$-dimensional simplicial sphere with~$k$ vertices, i.e. the $k$-cycle graph.

The following lemma gathers known facts on flag homology spheres.

\begin{lemma}\label{lem:elementary}
Let $\Delta$ be a $(d-1)$-dimensional flag homology sphere, $f\in\Delta$, $e\in \edge(\Delta)$, and $\Gamma$ be an equator of $\Delta$.
\begin{enumerate}[label=\roman{enumi}),ref=\ref{lem:elementary}~\roman{enumi})]
 \item The link $\lk_\Delta(f)$ is a flag induced homology sphere. \label{lem:link_flag}
 \item The contraction $\Delta^{\downarrow e}$ is a flag homology sphere if and only if $e$ is not contained in an induced $4$-cycle. \label{lem:contract_flag}
 \item If $d\geq 3$ and the link of every vertex of $\Delta$ is an octahedral sphere, then $\Delta$ is an octahedral sphere. \label{lem:octahedral_links}
 \item The deletion $\del_{\Delta}(\Gamma)$ is homologically equivalent to a $0$-dimensional sphere, i.e. \label{lem:alex_dual}
 \[
 \widetilde{H}_{*}(\del_{\Delta}(\Gamma),\mathbb{K}) \cong \widetilde{H}_{*}(S^0,\mathbb{K}),
 \]
 separated in $\Delta$ by $\Gamma$.
\end{enumerate}
\end{lemma}

\begin{proof}
i) Let $L:=\lk_\Delta(f)$ and $g\in \Delta$ such that $g\subseteq \verti(L)$.
Since every vertex of $g$ joined with $f$ is a face of $\Delta$ then the $1$-skeleton of $<g>\join <f>$ is in $\Delta$.
As the $1$-skeleton of $<g>\join <f>$ is a clique, $g\cup f$ is a face of $\Delta$ by flagness.
Therefore, $g$ is also a face of $L$, thus $L$ is induced and its flagness follows.

ii)
The ``only if'' direction is trivial.
For the other direction, assume $e$ is not contained in any  induced $4$-cycle.
Nevo and Novinsky proved that for any homology sphere $\Delta$, its edge contraction is again a  homology sphere if the link of the contracted edge $e$ is the intersection of the links of its two vertices \cite[Proposition~2.3]{nevo_characterization_2011}.
Since this is true for flag complexes, $\Delta^{\downarrow e}$ is a homology sphere.
Finally, the fact that flagness is preserved was proved by Lutz and Nevo, in the proof of  \cite[Corollary~6.2]{lutz_stellar_2016}.

iii) See \cite[Lemma~5.3]{nevo_gamma_2011} for a proof.

iv) Since $\Delta$ is a flag pseudomanifold without boundary, each facet of $\Gamma$ is contained in exactly two facets of $\Delta$, and neither of them have all their vertices in $\Gamma$.
Hence $\del_{\Delta}(\Gamma)$ contains at least two vertices, which are not connected by an edge since $\Delta$ is flag and $\Gamma$ is an equator of $\Delta$.
By Alexander duality on $\Delta$ and $\Gamma$, the result follows.
\end{proof}

For general homology spheres,
given a vertex $v$ of a homology sphere $\Delta$, the \defn{vertex split} of $v$ in~$\Delta$ along an equator $J$ of $\lk_\Delta(v)$ is the simplicial complex
\[
 \Delta^{v\vee J} := \del_\Delta(v) \cup \left(<\{v^+,v^-\}>\join J \right) \cup \left(<\{v^+\}>\join \lk_\Delta(v)_J^+ \right) \cup \left(<\{v^-\}>\join \lk_\Delta(v)_J^-\right),
\]
where $v^+,v^-$ are distinct vertices not in $\Delta$.
A special case of vertex split is when the equator $J$ is the link of an edge $e$ containing $v$, in which case it is the \defn{stellar subdivision} $\Delta^{\uparrow e}$ of that edge.
Note that for an edge $uv=e\in \Delta$, we have $(\Delta^{\uparrow e})^{\downarrow v^+v^-}=\Delta$.

\subsection{$f$- and $h$-vectors}

We introduce here basic notions on facial enumeration of simplicial complexes, see \cite{stanley_combinatorics_1996} for more details.

Given a pure simplicial complex $\Delta$ of dimension $d-1$, the $f$-vector $(f_0,f_1,\dots, f_{d})$ records the number $f_i$ of faces of cardinality $i$, for $0\leq i\leq d$, where $f_{0}=1$.
One may inscribe the $f$-vector in a polynomial
\[
 f^\Delta(z) := \sum_{i=0}^{d}f_iz^i.
\]
The shift $z\rightarrow z-1$ of the reciprocal polynomial of $f^\Delta(z)$ is called the $h$-vector of $\Delta$:
\[
 h^\Delta(z) := (z-1)^d f^\Delta((z-1)^{-1}) = \sum_{i=0}^dh_iz^i,
\]
whose coefficients are the linear combinations
\[
 h_k= \sum_{i=0}^k(-1)^{k-i}\binom{d-i}{d-k}f_{i},
\]
for $0\leq k\leq d$.

When $\Delta$ is a homology sphere, the generalized Dehn-Sommerville relations \cite{klee_combinatorial_1964} imply that the $h$-polynomial is self-reciprocal, that is
\[
 h^\Delta(z) = z^{d} h^\Delta(z^{-1}),
\]
and the sequence of coefficients of $h^\Delta(z)$ forms a palindromic sequence.
Once the $h$-vector is defined and seen to be palindromic, it is usual to rewrite the $h$-vector as the $g$-vector recording the difference of the consecutive entries in the first half of the $h$-vector: $g_0=1, g_i=h_i-h_{i-1}$ for $1\leq i\leq \lfloor d/2 \rfloor$.
It turns out that expressing upper and lower bounds on the number of faces can be expressed easily in terms of the entries of the $g$-vector and further other important combinatorial and geometric properties are cast into the $g$-vector.
See \cite{stanley_combinatorics_1996,ziegler_lectures_1995,murai_generalized_2013} for more details on the extremal values of $g$ and $h$-vectors.

\subsection{$\gamma$-vector of flag homology spheres}

In the particular case of flag homology spheres, another vector is conjectured to play a similar role as the $g$-vector.
Consider the $h$-polynomial of a flag homology sphere $\Delta$ of dimension~$d-1$.
By the Dehn-Sommerville relations,
 we can express $h^{\Delta}(z)$ in the basis
  \[
 \gimel_i:=z^i(1+z)^{d-2i},
\]
for $0\leq i \leq \lfloor \frac{d}{2}\rfloor$, and write it as
\[
 h^\Delta(z) = \sum_{i=0}^{\lfloor \frac{d}{2}\rfloor} \gamma_i^{\Delta} \gimel_i.
\]
Gal's $\gamma$-polynomial is then
\[
 \gamma^\Delta(t):= \sum_{i=0}^{\lfloor \frac{d}{2}\rfloor} \gamma_i^{\Delta}t^i.
\]
Let us emphasize that we switched the variable from $z$ to $t$; making it clear that a monomial $t^i$ in the $\gamma$-vector gives a shifted binomial power $z^i(1+z)^{d-2i}$ in the polynomial ring $\mathbb{Z}[z]$.
When the simplicial complex is clear from the context, we abuse language and simply write $\gamma_i$.
Gal conjectured that when $\Delta$ is a flag homology sphere, the coefficients $\gamma_i$ are nonnegative \cite[Conjecture~2.1.7]{gal_real_2005}.
Lately, this conjecture attracted a lot of attention.
It implies the Charney--Davis conjectures \cite{charney_euler_1995}.
For a recent survey concerning this topic, see \cite[Chapter~8-10]{petersen_eulerian_2015}.

The following basic lemma, proved by a direct computation, expresses the first entries of the $\gamma$-vector in terms of the $f$-vector.

\begin{lemma}\label{lem:gamma_values}
Let $\Delta$ be a $(d-1)$-dimensional
homology sphere
with $f_1$ vertices and $f_2$ edges. The entries of the $\gamma$-vector $\gamma^\Delta(t)$ satisfy
\begin{align*}
 \gamma_0 = 1, \qquad \gamma_1 = f_1 -2d, \qquad \gamma_2 = f_2 -(2d-3)f_1 +2d(d-2).
\end{align*}
Further,
\[
 \alpha+\gamma_2 = \frac{\gamma_1(\gamma_1+5)}{2}+d,
\]
where $\alpha$ is the number of missing edges of $\Delta$.
\end{lemma}

As a consequence, proving that $\gamma_2$ is non-negative is equivalent to proving that $\alpha\leq \frac{\gamma_1(\gamma_1+5)}{2}+d$.
Here are lemmas containing known facts on $\gamma$-vectors used later, see \cite[Section~5.2]{nevo_higher_2007}, \cite[Lemma~2.3]{aisbett_gamma_2012}, \cite[Section~5]{nevo_gamma_2011}, \cite[Corollary~1]{volodin_cubical_2010}.

\begin{lemma}\label{lem:gamma_eqns}
Let $\Delta$ be a $(d-1)$-dimensional flag homology sphere and $\Delta^{\downarrow e}$ be the contraction of an edge $e$ of $\Delta$ not contained in an induced $4$-cycle.
\begin{enumerate}[label=\roman{enumi}),ref=\ref{lem:gamma_eqns}~\roman{enumi})]
 \item The $\gamma$-polynomials $\gamma^\Delta(t)$ and $\gamma^{\Sigma\Delta}(t)$ are equal. \label{lem:gamma_suspension}
 \item The $\gamma$-polynomials of $\Delta$ and $\Delta^{\downarrow e}$ satisfy \label{lem:gamma_vertsplit}
\[
 \gamma^{\Delta}(t) = \gamma^{\Delta^{\downarrow e}}(t) +t\gamma^{\lk_\Delta(e)}(t).
\]
\end{enumerate}
\end{lemma}

\begin{lemma}[{\cite[Lemma~2.1.14]{gal_real_2005},\cite{meshulam_domination_2003}}]\label{lem:gamma_zero}
Let $\Delta$ be a $(d-1)$-dimensional flag homology sphere.
The $\gamma$-vector of $\Delta$ is $\gamma^\Delta(t)=1$ if and only if $\Delta$ is an octahedral sphere.
\end{lemma}

\section{Antipode number, polar size and upper bound for $\gamma_2$}\label{sec3:gamma2}

\subsection{Definitions}

In this section, we formalize the notions of \emph{antipode number} and \emph{polar size} of flag homology spheres and give some elementary properties related to their $\gamma$-vectors.
This notion was used previously, for instance in \cite[Section~5]{nevo_gamma_2011}.

\begin{definition}
Let $\Delta$ be a flag homology sphere and $v\in \verti(\Delta)$.
An \defn{antipode} of $v$ is a vertex $w\in \verti(\Delta)$ such that $\{v,w\}$ is a missing edge of $\Delta$.
The \defn{antipode number} $\iota(v)$ is the number of antipodes of $v$.
The \defn{polar size} $\pi_{\Delta}$ of $\Delta$ is the minimal antipode number $\pi_{\Delta}:=\min_{v\in \verti(\Delta)}\iota(v)$ among vertices of $\Delta$.
\end{definition}

The sum of the antipode numbers of a flag homology sphere divided by 2 is its number of missing edges, i.e.
\[
 \binom{f_1}{2} = f_2 +\frac{1}{2}\sum_{v\in \verti(\Delta)}\iota(v).
\]
Therefore, studying flag homology spheres via their antipode numbers seems to be a natural way to understand their
facial enumeration.
The following lemma gives information about the possible polar sizes for flag homology spheres.

\begin{lemma}\label{lem:polar_values}
Let $\Delta$ be a $(d-1)$-dimensional flag homology sphere on $2d+\ell$ vertices.
\begin{enumerate}[label=\roman{enumi}),ref=\ref{lem:polar_values}~\roman{enumi})]
 \item $\gamma_1^{\lk_\Delta(v)}=\ell-\iota(v)+1$, for all $v\in \verti(\Delta)$, \label{lem:polar_values_i}
 \item $\pi_\Delta\in\{1,\dots,\ell+1\}$, \label{lem:polar_values_ii}
 \item $\pi_\Delta=1$ if and only if $\Delta$ is a suspension of a $(d-2)$-dimensional flag homology sphere, \label{lem:polar_values_iii}
 \item if $d\geq 3$, then $\pi_\Delta=\ell+1$ if and only if $\ell=0$, i.e., $\Delta$ is an octahedral sphere. \label{lem:polar_values_iv}
\end{enumerate}
\end{lemma}

\begin{proof}
 i) The number of vertices in $\lk_\Delta(v)$ is $(2d+\ell)-(\iota(v)+1)$ and the dimension of the homology sphere $\lk_\Delta(v)$ is $d-2$. By Lemma~\ref{lem:gamma_values} we get the desired value.

 ii) By Lemma~\ref{lem:alex_dual}, $\pi_\Delta\geq 1$.
 Besides, as the links of vertices of $\Delta$ have to be flag homology spheres of dimension $d-2$ and need at least $2(d-1)$ vertices, the polar size $\pi_\Delta$ is bounded above by $\ell+1$.

 iii) This is straightforward, using Alexander duality.

 iv) If $\pi_\Delta=\ell+1$, then the links of all vertices of~$\Delta$ are octahedral spheres by i) and Lemma~\ref{lem:gamma_zero}.
 Then by Lemma~\ref{lem:octahedral_links}, $\Delta$ is an octahedral sphere, which implies that $\ell=0$.
 The reverse direction is straightforward from the definitions.
\end{proof}

\subsection{Polar size and $\gamma$-vectors}

We state and give a short proof of a result generalized in Theorem~\ref{thm:gamma2_bounded}.

\begin{lemma}[{\cite[Proposition~5.4]{nevo_gamma_2011}}]\label{lem:susp_pentagon}
Let $\Delta$ be a $(d-1)$-dimensional flag homology sphere.
If $\gamma_1^\Delta=1$, then $\gamma^\Delta(t)=1+t$ and $\Delta= \Sigma^{d-2}\join C_5$.
\end{lemma}

\begin{proof}
Assume that $\gamma_1^\Delta=1$.
The statement is relevant only when $d>1$.
If $d=2$, then $1=\gamma_1^\Delta=f_1^\Delta-2d$, and $\Delta\cong C_5$ and $\gamma_\Delta(t)=1+t$.
For $d>2$, by Lemma~\ref{lem:polar_values_ii}, $\pi_\Delta\in\{1,2\}$.
By Lemma~\ref{lem:polar_values_iv}, $\pi_\Delta=2$ if and only if $\Delta$ is an octahedral sphere which is impossible.
Hence $\pi_\Delta=1$ and by Lemma~\ref{lem:gamma_suspension} and~\ref{lem:polar_values_iii},~$\Delta$ is the suspension of a lower dimensional flag homology sphere with the same $\gamma$-vector and this finishes the proof by induction.
\end{proof}

\begin{lemma}\label{lem:iota_equals_2}
Let $\Delta$ be a $(d-1)$-dimensional flag homology sphere on $2d+\ell$ vertices and $v\in \verti(\Delta)$ such that $\iota(v)=2$. Then,
\begin{enumerate}[label=\roman{enumi}),ref=\ref{lem:iota_equals_2}~\roman{enumi})]
 \item The two antipodes $x$ and $y$ of $v$ form an edge not contained in an induced $4$-cycle, and \label{lem:iota_equals_2_edge}
 \item $\gamma^\Delta(t)=\gamma^L(t)+t\gamma^J(t)$, \label{lem:iota_equals_2_eqn}
\end{enumerate}
where $L:=\lk_\Delta(v)$ and $J:=\lk_\Delta(\{x,y\})$.
\end{lemma}

\begin{proof}
i) By Lemma~\ref{lem:alex_dual}, the antipodes $x$ and $y$ have to be connected by an edge.
Assume that the edge $\{x,y\}$ is contained in a $4$-cycle $(x,y,s,t)$, where $s\in\lk_\Delta(y)$  and $t\in\lk_\Delta(x)$.
The link $J$ of the edge $\{x,y\}$ in $\Delta$ is an equator of~$L$.
By the Jordan--Alexander theorem, the homology sphere $L$ decomposes into two hemispheres intersecting in $J$, $L=L_J^+\cup_J L_J^-$.
Further, $L_J^+$ and $L_J^-$ are the only hemispheres contained in $L$ with boundary $J$.
On the other hand, the hemispheres $\del_{\lk_{\Delta}(y)}(x)$ and $\del_{\lk_{\Delta}(x)}(y)$ are different homology balls, both contained in $L$ with boundary $J$, so w.l.o.g. the former is $L_J^+$ and the latter is $L_J^-$; thus $s\in L_J^+$ and $t\in L_J^-$.

Because $v$ is not connected to $x$ nor to $y$, $s$ and $t$ have to be vertices in~$L$ and since links are induced complexes by Lemma~\ref{lem:link_flag}, the edge between $s$ and $t$ belongs to $L$. As $J$ separates $L_J^+$ from $L_J^-$, one of $s,t$ must belong to $J$.
Thus, the $4$-cycle $(x,y,s,t)$ is not induced.

ii) By Lemma~\ref{lem:contract_flag}, $\Delta^{\downarrow e}$ is a flag homology sphere and by Lemma~\ref{lem:gamma_vertsplit}, where $e = \{x,y\}$, we get
\[
 \gamma^\Delta(t)=\gamma^{\Delta^{\downarrow e}}(t)+t\gamma^J(t)
\]
Observe that $\Delta^{\downarrow e}\cong \Sigma L$, so by Lemma~\ref{lem:gamma_suspension} the assertion follows.
\end{proof}

\begin{theorem}\label{thm:polar_values}
Let $\Delta$ be a $(d-1)$-dimensional flag homology sphere with polar size~$\pi_\Delta >1$, and $v_0\in \verti(\Delta)$ with $\iota(v_0)=\pi_\Delta$.
\begin{enumerate}[label=\roman{enumi}),ref=\ref{thm:polar_values}~\roman{enumi}),itemsep=-0ex]
\item If $\lk_{\Delta}(v_0)\cong\Sigma\Gamma$, where $\Gamma$ is a $(d-3)$-dimensional flag homology sphere, then $\Delta\cong \Gamma \join C_{\pi_\Delta+3}$. Further,
    $\pi_\Delta-1$ divides the highest degree coefficient of $\gamma^\Delta(t)$. \label{thm:extract_cycle}

\item If $d\geq 4$ and $\pi_\Delta=\gamma_1^\Delta$, then $\Delta= C_5\join C_{\pi_\Delta+3}$ (and $d=4$).
    \label{thm:iota_equals_ell}
\end{enumerate}
\end{theorem}

\begin{proof}
 i) Let $\lk_{\Delta}(v_0)$ be the suspension of $\Gamma$ over the two vertices $u,w$, and denote by $\{w_1,\dots,w_{\pi_\Delta}\}$ the antipodes of $v_0$.
 As $\{u,w\}$ is not an edge of $\Delta$ and $u$ is connected to $v_0$ and to all vertices of $\Gamma$, $u$ requires at least $\pi_\Delta-1$ antipodes in $\{w_1,\dots,w_{\pi_\Delta}\}$. However, $u$ is contained in a facet containing an antipode of $v_0$, so w.l.o.g. $uw_1$ is an edge of $\Delta$. Further, the complex $\Gamma=\lk_{\Delta}(uv_0)$ is an \emph{induced} subcomplex of $\lk_{\Delta}(u)$, thus,
 by Lemma~\ref{lem:alex_dual}, $\lk_{\Delta}(u)$ is the suspension of $\Gamma$ by $v_0$ and $w_1$.
 The same argument can be applied sequentially with $w_1$ to get $\lk_{\Delta}(w_1)=\{u,w_2\}\join \Gamma$ and so on, to finally get
 that $\Delta$ contains the $(d-1)$-homology sphere $\Gamma\join (v_0,u,w_1,w_2,\dots,w_i,w)$, for some $i$. Hence, by Alexander duality, $i=\pi_\Delta$ and thus $\Delta\cong\Gamma\join (v_0,u,w_1,\dots,w_{\pi_\Delta},w)\cong \Gamma\join C_{\pi_\Delta+3}$.
Then $\gamma^{\Delta}(t)=\gamma^{\Gamma}(t)\gamma^{C_{\pi_\Delta+3}}(t)=
\gamma^{\Gamma}(t)(1+(\pi_\Delta-1)t)$, so $\pi_\Delta-1$ divides the top coefficient of $\gamma^{\Delta}(t)$.

 ii)
 Let $L:=\lk_{\Delta}(v_0)$.
 By Lemma~\ref{lem:polar_values_i}, $\gamma_1^L=\gamma_1^\Delta-\pi_\Delta+1=1$, consequently by Lemma~\ref{lem:susp_pentagon} $L= \Sigma^k\join C_5$, for some $k\geq0$.
 Since $d\geq4$, $k$ is at least $1$, i.e. $L$ is a suspension.
 By part~i), $\Delta= \Sigma^{k-1}\join C_5 \join C_{\gamma_1^\Delta+3}$, and as $\pi_{\Delta}>1$ we conclude $k=1$.
\end{proof}

The following corollary extends Lemma~\ref{lem:octahedral_links}. 

\begin{corollary}
 Let $d\geq 3$, $\Delta$ be a $(d-1)$-dimensional flag homology sphere, and $\Gamma$ a simplicial complex such that $\lk_\Delta(v)\cong\Sigma\Gamma$ for all $v\in V(\Delta)$.
 If $\Gamma$ is not an octahedral sphere, then there exist integers $\ell\geq2$ and $k\geq 5$ such that $d=2\ell$, $\Delta\cong\join^\ell C_k$, and $\Gamma\cong\join^{\ell-1}C_k$.
\end{corollary}
\begin{proof}
The complex $\Delta$ has no suspension vertex; else every vertex would be a suspension vertex, forcing $\Delta$ be octahedral, and thus $\Gamma$ be octahedral as well, giving a contradiction. Thus, $\pi_{\Delta}>1$.

By Theorem~\ref{thm:extract_cycle}, $\Delta\cong\Gamma\join C_{\pi_\Delta+3}$. Let $k=\pi_\Delta+3\geq 5$. As all vertex links in $\Delta$ are isomorphic,~$\Gamma$ must factors as $\Gamma\cong C_k\join\Lambda$ for some $\Lambda$.
By repeating this argument, $\Delta$ factors as $\Delta\cong \join^\ell C_k$ for some $\ell\geq2$. Thus $d=2\ell$ and $\Gamma\cong\join^{\ell-1}C_k$.
\end{proof}

\subsection{Upper bound on $\gamma_2$}

The following theorem generalizes Lemma~\ref{lem:susp_pentagon} of Nevo and Petersen, which treated the case $\ell=1$.

\begin{theorem}\label{thm:gamma2_bounded}
Let $\ell\geq 0$.
If $\Delta$ is a $(d-1)$-dimensional flag homology sphere on $2d+\ell$ vertices, then $\gamma_2\leq \binom{\ell}{2}$.
Furthermore, the equality $\gamma_2= \binom{\ell}{2}$ holds if and only if
\[
	\Delta = \Sigma^{m}\ast^{\ell}C_5,
\]
for some $m\geq 0$.
\end{theorem}

We present two lemmas, which may be of independent interest, to be used in the proof. Assertion~\ref{lem:u2facets_i} of Lemma~\ref{lem:u2facets} is known \cite{meshulam_domination_2003,gal_real_2005}.

\begin{lemma}\label{lem:u2facets}
Let $S$ be a flag homology sphere.
\begin{enumerate}[label=\roman{enumi})]
 \item If $T_1\in S$ is a facet, then exists another facet $T_2\in S$ disjoint from $T_1$.\label{lem:u2facets_i}
 \item If $S$ is not a suspension, and $s\in S$ is a vertex, then there exist two disjoint facets $T_1,T_2\in S$ not containing $s$.\label{lem:u2facets_ii}
\end{enumerate}
\end{lemma}

\begin{proof}
Proceed by induction on the dimension $t$ of $S$, where for $t=1$ the claims are clear. Let $t>1$.

\ref{lem:u2facets_i} Let $T_1$ be a facet of $S$ and $s'\in T_1$ a vertex.
By hypothesis, for $T'_1:=T_1-\{s'\}$ there exists another facet $T'_2$ of $S':=\lk_{S}(s')$ disjoint from $T'_1$.
By the Jordan--Alexander theorem there exists a facet $T_2\in S$ such that $T_2=T'_2 \sqcup \{y\}$ where $y$ is in the connected component of $S-S'$ not containing $s'$. This proves \ref{lem:u2facets_i}.

\ref{lem:u2facets_ii} Let $s\in S$ and $s'$ be an antipode of $s$.
As $S$ is not a suspension, there exists a facet $T'_2$ of $S':=\lk_{S}(s')$ such that $\{s\}\cup T'_2$ is not a face of $S$.
By \ref{lem:u2facets_i}, there is a facet $T'_1$ of $S'$ disjoint from $T'_2$.
As in part \ref{lem:u2facets_i}, there are disjoint facets $T_1=T'_1\cup \{s'\}$ and $T_2=T'_2\cup \{y\}$ where $y$ is not~$s'$, and by the choice of $T'_2$, $y$ is not $s$.
This proves \ref{lem:u2facets_ii}.
\end{proof}

\begin{lemma}\label{lem:polar_size_3_on_gamma2}
Let $\ell\geq 0$.
If $\Delta$ is a $(d-1)$-dimensional flag homology sphere on $2d+\ell$ vertices with $\pi_\Delta\geq 3$, then $\gamma_2<\binom{\ell}{2}$.
\end{lemma}

\begin{proof}
The proof is done by sharpening a naive upper bound.

From Lemma~\ref{lem:gamma_values}, we have
\begin{equation}
\gamma_2=f_2-(2d-3)f_1+2d(d-2). \tag{*}\label{eqn:e}
\end{equation}
Therefore $\gamma_2^{\Delta}<\binom{\ell}{2}$ if and only if
\begin{equation}
f_2\leq \binom{f_1}{2}-3f_1+5d-1=:C. \tag{**}\label{eqn:ee}
\end{equation}
Let $v_0\in\Delta$ such that $\iota(v_0)=\pi_{\Delta}:=a+1$ and let $L:=\lk_\Delta(v_0)$.
Then $\gamma_1^L=\ell-a$ and by induction, $\gamma_2^{L}\leq \binom{\ell-a}{2}$.

{\bf Case 1) $L$ is a suspension.}
Say $L=\Sigma\Gamma$.
By Theorem \ref{thm:polar_values}, $\Delta\cong \Gamma \join C_{3+\pi_{\Delta}}$.
Then $\gamma_2^\Delta=a(\ell-a)+\gamma_2(\Gamma) \leq a(\ell-a)+\binom{\ell-a}{2}$.
As $a=\pi_{\Delta}-1\geq 2$, we conclude $\gamma_2^\Delta< \binom{\ell}{2}$ as desired.

{\bf Case 2) $L$ is not a suspension.}
By $\eqref{eqn:e}$, $$f_2^L\leq \binom{\ell-a}{2} + (2d-5)(2d+\ell-a-2)-2(d-1)(d-3)=:N_1.$$
Let $I:=\del_\Delta(\st_\Delta(v_0))$ and  decompose the set of edges $\edge(\Delta)$ as
\[ \edge(\Delta)=
 \edge(L) \sqcup \left\{\{v_0,u\}: u\in \verti(L)\right\} \sqcup \edge(I) \sqcup \left\{\{u,w\}\in\Delta: u\in \verti(L), w\in \verti(I)\right\}.
\]
From this decomposition, we get a naive upper bound for the number of edges in $\edge(\Delta)$:
\[
 |\edge(\Delta)| \leq N_1 + (2d+\ell-a-2) + \binom{a+1}{2} + (a+1)(2d+\ell-a-2) =: B.
\]
Comparing this $B$ with the right-hand side $C$ of \eqref{eqn:ee}, a direct computation shows that $B=C+2a+1$.
To complete the proof, it suffices to improve the estimate on the cardinality of $\edge(L,I):=\left\{\{u,w\}\in\Delta: u\in \verti(L), w\in \verti(I)\right\}$  by showing
$$|\edge(L,I)| \leq |\verti(L)||\verti(I)|-2a-1.$$

As $|\verti(I)|>1$, clearly there exists a missing edge $uw$ of $\Delta$ with $u\in L$ and $w\in I$.
By Lemma~\ref{lem:u2facets}, there are 2 disjoint facets $F_1$ and $F_2$ in $L$ neither of which contains $u$.
Since~$\Delta$ is a pseudomanifold, $F_1$ (respectively $F_2$) is contained in a unique facet involving a vertex $u_{F_1}$ (resp.~$u_{F_2}$, these vertices may coincide) in $I$.
For each $x\in \verti(I)\setminus \{u_{F_1}\}$, there exists a missing edge between $x$ and a vertex of $F_1$.
The same argument can be applied to $u_{F_2}$ and $F_2$ to save $2a$ edges, none of which contains $u$, thus $|\edge(L,I)| \leq |\verti(L)||\verti(I)|-2a-1$, completing the proof.
\end{proof}

\begin{proof}[Proof of Theorem~\ref{thm:gamma2_bounded}]

We prove the theorem by induction on $\ell$.
By Lemma~\ref{lem:gamma_zero}, if $\ell=0$, then~$\Delta$ is an octahedral sphere and $\gamma_i=0$ for all $i\geq 2$.
By Lemma~\ref{lem:susp_pentagon}, $\ell=1$ if and only if~$\Delta$ is a repeated suspension over a pentagon, and $\gamma_2=\binom{1}{2}=0$ holds.

Hence assume $\ell>1$ and that the result holds for all $\Delta$ with $\gamma_1^{\Delta}< \ell$.
By Lemma~\ref{lem:gamma_suspension}, we can remove pairs of suspension vertices leaving the $\gamma$-vector unchanged.
So, assume that~$\Delta$ has no pair of suspension vertices, so that the polar size of $\Delta$ satisfies $\pi_\Delta\geq 2$.
We distinguish two cases: $\pi_{\Delta}=2$ or $\pi_{\Delta}>2$.

{\bf Case $\pi_{\Delta}=2$.} Let $v_0\in\Delta$ be such that $\iota(v_0)=2$, and denote the antipodes of $v_0$ by $x$ and $y$.
By Lemma~\ref{lem:alex_dual} the edge $e:=\{x,y\}$ exists. Let $J:=\lk_\Delta(e)$ and $L:=\lk_\Delta(v_0)$.
By Lemma~\ref{lem:iota_equals_2_eqn}, we get
\[
 \gamma^\Delta(t)=\gamma^{L}(t)+t\gamma^J(t).
\]
By Lemma~\ref{lem:polar_values_i}, $\gamma_1^{L}=\ell-1$.
The dimension of~$J$ is $2$ less than of $\Delta$, and the number of vertices decreases by at least 5 from $\Delta$: we remove $v_0$, the two antipodes of $v_0$ and at least two vertices from $L$.
Therefore $\gamma_1^J\leq \ell-1$, using Lemma~\ref{lem:gamma_values}.
By the induction hypothesis on $L$, $\gamma_2^{L}\leq \binom{\ell-1}{2}$.
We thus get $\gamma_2^{\Delta} = \gamma_2^{L}+\gamma_1^J\leq \binom{\ell-1}{2}+(\ell-1)=\binom{\ell}{2}$.

We now prove the equality case.
Assume that $\gamma_2^{\Delta}=\binom{\ell}{2}$.
By the above discussion, we have $\gamma_2^{\Delta^{\downarrow e}}=\binom{\ell-1}{2}$ and $\gamma_1^J=\ell-1$.
By induction on $\ell$, $\Delta^{\downarrow e}\cong\Sigma^{k}\ast^{\ell-1}C_5$, with $k\geq 1$, as $\Delta^{\downarrow e}$ is the suspension $\Sigma L$.
Since $\gamma_1^L=\gamma_1^J=\ell-1$ and $J$ is $(d-3)$-dimensional, the number of vertices of $J$ is exactly $5$ less than of $\Delta$ and forces $L=\Sigma J$.
Whence $\Delta^{\downarrow e}\cong \Sigma^2 J$ and consequently, $J\cong\Sigma^{m}\ast^{\ell-1}C_5$, for some $m\geq 0$.
By Theorem~\ref{thm:extract_cycle} applied on the link $L=\Sigma J$, $\Delta\cong J\join C_5\cong \Sigma^{m}\ast^{\ell}C_5$.

{\bf Case $\pi_{\Delta}>2$.} By Lemma~\ref{lem:polar_size_3_on_gamma2} $\gamma_2^{\Delta}<\binom{\ell}{2}$ which implies the equality part as well since the polar size of $\Sigma^{m} \ast^{\ell} C_5$ is at most $2$.
\end{proof}

\section{Bounds on $\gamma_{\gamma_1}^{\Delta}$ and a flag analog of Perles' theorem}~\label{sec4:gamma_gamma1}

In this section, we present an analog of Perles' result on the boundedness of combinatorial types of $k$-skeleta of polytopes with fixed toric $g_1$.
This theorem is studied in detail in the survey article \cite{kalai_aspects_1994}.

Perles proved that for fixed $k$ and $b$, if $d$ is large enough then any $d$-polytope with toric $g_1=b$ has the same $k$-skeleton as some other $d$-polytope which is a pyramid over a lower dimensional polytope.
In the flag case, the role of a pyramid is played by suspension, as the following (stronger) assertion shows.

\begin{lemma}\label{lem:bound_dim}
Let $\Delta$ be a $(d-1)$-dimensional flag homology sphere on $2d+\ell$ vertices, with $\ell\geq 0$.
If $d\geq2\ell+1$, then $\Delta\cong\Sigma\Gamma$, where $\Gamma$ is a $(d-2)$-dimensional flag homology sphere.
\end{lemma}

\begin{proof}
The proof is by induction on $\ell$.
If $\ell=0$, by Lemma~\ref{lem:gamma_zero},~$\Delta$ is an octahedral sphere and the result follows.
Assume that the result holds for all flag homology spheres such that $\gamma_1<\ell$.
We prove the contraposition: if $\Delta\not\cong\Sigma\Gamma$, then $d<2\ell+1$.
The polar size of $\Delta$ satisfies $\pi_\Delta\geq 2$ by Lemma~\ref{lem:polar_values_iii}.
Let $v_0$ be a vertex of $\Delta$ such that $\iota(v_0)=\pi_\Delta$ and $L:=\lk_\Delta(v_0)$.
By Lemma~\ref{lem:polar_values_i}, $\gamma_1^L=\ell-\pi_\Delta+1\leq \ell-1$.

If $L$ is not a suspension, by the induction hypothesis $d-1<2\gamma_1^L+1$, or equivalently $d<2(\gamma_1^L+1)\leq 2\ell$.
Otherwise $L$ is a suspension, by Theorem~\ref{thm:extract_cycle}, $\Delta\cong J\join C_{\pi_\Delta+3}$, where $J$ is a $(d-3)$-dimensional flag homology sphere and $\gamma_1^J\leq \ell -1$.
Since~$\Delta$ is assumed not to be a suspension, $J$ is also not a suspension.
By the induction hypothesis on $J$, $d-2<2\gamma_1^J+1$ or equivalently $d<2\gamma_1^J+3=2(\gamma_1^J+1)+1\leq 2\ell +1$.
\end{proof}

\begin{theorem}\label{thm:bound_index}
Let $\ell\geq 0$ and $d\geq2$.
If $\Delta$ is a $(d-1)$-dimensional flag homology sphere on $2d+\ell$ vertices, then $0\leq \gamma_{\ell}\leq 1$, and $\gamma_j=0$ for all $j>\ell$.
Further, $\gamma_\ell=1$ if and only if $\Delta\cong \Sigma^m\join^\ell C_5$, for a certain $m\geq 0$.
\end{theorem}

\begin{proof}
By Lemma~\ref{lem:bound_dim}, if $d-1\geq 2\ell$, we remove the suspension vertices until we get down to a homology sphere of dimension at most $2\ell-1$.
Therefore, $\gamma_j=0$ for all $j>\ell$.
The proof for $\gamma_\ell$ is by induction on $\ell$.
If $\ell=0$, by Lemma~\ref{lem:gamma_zero},~$\Delta$ is an octahedral sphere $\Sigma^m\join^0 C_5$, $\gamma_0=1$, and $\gamma_i=0$ for all $i\geq 1$.
So assume that the result holds for all flag homology spheres with $\gamma_1<\ell$.

If the dimension of $\Delta$ satisfies $d-1\leq 2\ell-2=2(\ell-1)$, by the above $\gamma_\ell=0$.
Additionally, the sphere~$\Delta$ can not be isomorphic to $\Sigma^m\join^\ell C_5$ since the latter has dimension $2\ell+m-1\geq 2\ell-1>2\ell-2\geq d-1$.
Otherwise, $d=2\ell$.
As $\Delta$ does not have suspension vertices, let $v_0$ be a vertex of~$\Delta$ with $\iota(v_0)=\pi_\Delta\geq 2$ and denote $L:=\lk_\Delta(v_0)$.
The conditions of Lemma~\ref{lem:bound_dim} apply on $L$, so that~$L$ is a suspension $\Sigma J$, where~$J$ is a $(2\ell-3)$-dimensional flag homology sphere with $\gamma_1^J\leq \ell-1$.
By Theorem~\ref{thm:extract_cycle}, $\Delta\cong J \join C_{\pi_\Delta+3}$.
Therefore, $\gamma^\Delta(t)=(1+(\pi_\Delta-1)t)\gamma^J(t)$ and $\gamma_\ell^\Delta$ is equal to $(\pi_\Delta-1)\gamma_{\ell-1}^J$.

If $\gamma_1^J<\ell-1$, then $\gamma_j^J=0$ for all $j\geq \ell-1$ by the above.
Therefore $\gamma_\ell^\Delta=0$.
Otherwise $\gamma_1^J=\ell-1$ which implies that $\pi_\Delta=2$ and $\Delta\cong J \join C_{5}$.
Thus $\gamma_\ell^\Delta=\gamma_{\ell-1}^J$.
As $\gamma_1^J=\ell-1$, by induction on $\ell$, $0\leq\gamma_{\ell-1}^J\leq 1$, and $\gamma_{\ell-1}^J=1$ if and only if $J\cong \Sigma^m\join^{\ell-1} C_5$, for some $m\geq 0$.
Thus $\Delta\cong \Sigma^m\join^\ell C_5$ completing the proof.
\end{proof}

Along similar lines to those of Lemma~\ref{lem:polar_size_3_on_gamma2}, the following corollary gives another relation between polar sizes $\pi_\Delta\geq 3$ and entries of the $\gamma$-vector.

\begin{corollary}\label{cor:polar_size_3_degree_gamma}
If $\Delta$ is a $(d-1)$-dimensional flag homology sphere on $2d+\ell$ vertices with $\pi_\Delta\geq 3$, and $d\geq 3$, then $\gamma_{j}^\Delta=0$, for all $j\geq \ell-\pi_\Delta+2$.
Further, if $d\geq 2(\ell-\pi_\Delta+2)$, then $\pi_\Delta-1$ divides the highest degree coefficient of $\gamma^\Delta(t)$.
\end{corollary}
\begin{proof}
Having $\pi_\Delta\geq 3$ and $d\geq 3$ implies that $\ell\geq 3$ by Lemma~\ref{lem:polar_values_ii} and~\ref{lem:polar_values_iv}.
If $d<2(\ell-\pi_\Delta+2)$ then $\gamma^\Delta_j=0$, for all $j\geq \ell-\pi_\Delta+2$.
Therefore, assume that $d\geq 2(\ell-\pi_\Delta+2)$.
Let $v_0\in\Delta$ be such that $\iota(v_0)=\pi_\Delta$, and denote $L:=\lk_\Delta(v_0)$.
The conditions of Lemma~\ref{lem:bound_dim} apply on~$L$ since $\pi_\Delta\geq3$, so that $L=\Sigma J$, where $J$ is a $(d-3)$-dimensional flag homology sphere.
By Theorem~\ref{thm:extract_cycle}, $\Delta\cong J\join C_{\pi_\Delta+3}$ and $\pi_\Delta-1$ divides the highest degree coefficient of $\gamma^{\Delta}(t)$.
Since $\gamma^J_1=\ell-\pi_\Delta+1$, the degree of $\gamma^J(t)$ is at most $\ell-\pi_\Delta$ (it can not be $\ell-\pi_\Delta+1$ by Theorem~\ref{thm:bound_index} and $\pi_{\Delta}\geq 3$), whence the degree of $\gamma^\Delta(t)$ is at most $\ell-\pi_\Delta+1$.
\end{proof}

The following result gives a ``flag'' analog to Perles' theorem \cite[Theorem~1.1]{kalai_aspects_1994}.

\begin{theorem}
Let $d,k\geq 1$, $b\geq 0$ and $f_{b,k}(d)$ be the number of combinatorial types of $k$-skeleta of flag homology spheres of dimension $d-1$ with $\gamma_1=b$.
Given $b$ and $k$, there exists a constant $c_{b,k}$ such that $f_{b,k}(d)\leq c_{b,k}$.
\end{theorem}

\begin{proof}
 If $d\geq 2b+1$, following Lemma~\ref{lem:bound_dim}, by removing suspension vertices, every combinatorial type of $k$-skeleta of flag homology spheres of dimension $d-1$ corresponds to a unique one of dimension $2b-1$.
 Indeed, the result is unique since suspension pairs are uniquely determined by their two vertices and these (unordered) pairs are pairwise disjoint, thus the order of removal of suspension vertices is not important.
 Therefore $f_{b,k}(d)\leq f_{b,k}(2b):=c_{b,k}$.
 Otherwise, $d\leq 2b$ and for each combinatorial type of $k$-skeleta of flag homology spheres of dimension $d-1$, one can create a combinatorial type of $k$-skeleta of flag homology spheres of dimension $2b-1$ by taking suspensions and each will be combinatorially different.
 Thus, $f_{b,k}(d)\leq f_{b.k}(2b):=c_{b,k}$.
\end{proof}

Theorem~\ref{thm:intro_d} follows seeing that $F_{b,k}(d)=\sum_{i=0}^bf_{i,k}(d)\leq (b+1)c_{b,k}$.

\section{Bounds on $\gamma_{\gamma_1-1}$ and forbidden $\gamma$-vectors}~\label{sec5:gamma_gamma1-1}

We start with an observation on the structure of equators in the maximizers of $\gamma_{\gamma_1}$ ($=1$), which plays a role in the characterization of extremal examples of maximizers of $\gamma_{\gamma_1 -1}$ provided $\gamma_{\gamma_1}=0$, given in the next theorem.

\begin{lemma}\label{lem:equators_type}
Let $m,k\geq0$ such that $m+k\geq 1$.
All equators of $\Sigma^m\join^k C_5$ are of the form $\Sigma^{m-1}\join^{k}C_5$ or $\Sigma^{m+1}\join^{k-1}C_5$, and every equator is the link of a vertex.
\end{lemma}

\begin{proof}
Let $\Gamma$ be an equator of $Z:=\Sigma^m\join^k C_5$.
For an induced subcomplex of $Z$ \emph{not} to be acyclic, from each pair of suspension vertices either both or none are in $\Gamma$, and from each $C_5$ either all vertices, none, or exactly two non-adjacent ones are in $\Gamma$.
Combined with the fact that the dimension of $\Gamma$ is $1$ less than of $Z$, either $\Gamma\cong \Sigma^{m-1}\join^{k}C_5$ or $\Gamma\cong \Sigma^{m+1}\join^{k-1}C_5$. In the former case $\Gamma$ is the link of a suspension vertex, and in the latter case $\Gamma$ is the link of a vertex in some induced $C_5$.
\end{proof}

\begin{theorem}\label{thm:gamma_ell-1}
Let $\ell\geq 2$.
If $\Delta$ is a $(d-1)$-dimensional flag homology sphere on $2d+\ell$ vertices, then $\gamma_{\ell-1}\in\{0,1,2,\ell\}$. Furthermore, if $\gamma_\ell=0$, then $0\leq \gamma_{\ell-1}\leq 2$ and the case $\gamma_{\ell-1}=2$ holds if and only if
\[
	\Delta = \Sigma^{m}\ast^{\ell-2}C_5\ast C_6=:\Upsilon_1,
\]
for some $m\geq 0$, called ``first type", or $\Delta$ is obtained by an edge-subdivision of
\[
	\Sigma^{m}\ast^{\ell-1}C_5,
\]
at an edge adjacent to a suspension vertex, called ``second type" and denoted by $\Upsilon_2$.
\end{theorem}

\begin{proof}
We prove the theorem by induction on $\ell$.
If $\gamma_1=\ell=2$, the first statement trivially holds.
It remains to prove that if $\gamma_2=0$ then $\Delta$ is the corresponding $\Upsilon_1$ or $\Upsilon_2$.
The other direction is obtained by direct computation.
Essentially, the proof of \cite[Proposition 5.5]{nevo_gamma_2011} gives this structural result; however, we elaborate here for completeness.
By Lemma~\ref{lem:gamma_suspension}, we can assume that~$\Delta$ has no pair of suspension vertices, i.e., $\pi_\Delta\geq 2$.
By Lemma \ref{lem:polar_values_ii}, $\pi_\Delta\in\{2,3\}$.
Assume $\pi_\Delta=2=\ell$.
If $d\geq 4$, by Theorem~\ref{thm:iota_equals_ell}, $\Delta\cong C_5\join C_5$, so that $\gamma_2^\Delta=1$ which contradicts the assumption.
Therefore $d\leq 3$.
If $d=2$, only the $6$-gon has $\gamma_1=2$ and this concludes this case.
So, assume $d=3$ and let $v_0\in\Delta$ be such that $\iota(v_0)=2$, and denote the antipodes of $v_0$ by $x$ and $y$.
Let $e:=\{x,y\}\in\Delta$, $J:=\lk_\Delta(e)$ and $L:=\lk_\Delta(v_0)$.
By Lemma~\ref{lem:iota_equals_2_eqn}, we get
\[
 \gamma^\Delta(t)=\gamma^{L}(t)+t\gamma^J(t).
\]
By Lemma~\ref{lem:polar_values_i}, $\gamma_1^{L}=1$ and by Lemma~\ref{lem:susp_pentagon} (or directly) $L$ is a pentagon.
Furthermore, from the above information (note that $J$ consists of 2 nonadjacent vertices) one sees that the $8$-vertices flag $2$-sphere
$\Delta$ is obtained by taking a suspension over $L$ ($v_0$ is in the suspension pair) followed by a vertex split in the other suspension vertex, which in this case is in fact an edge-subdivision with respect to an edge adjacent to the other suspension vertex (not $v_0$), i.e., $\Upsilon_2$.
Assume now that $\pi_\Delta=3=\ell+1$.
If $d\geq 3$, by Lemma~\ref{lem:polar_values_iv}, $\Delta$ is an octahedral sphere, contradicting our assumption.
If $d=2$, again only the $6$-gon has $\gamma_1=2$ and this concludes the base case.

Let $\ell>2$ and assume the result to be true for all homology spheres such that $\gamma_1<\ell$ and $\gamma_\ell^\Delta=0$.
Indeed, by Theorem~\ref{thm:bound_index}, if $\gamma_\ell^\Delta=1$, then $\Delta$ is isomorphic to $\Sigma^{m}\ast^{\ell}C_5$ for some $m\geq 0$ and $\gamma_{\ell-1}^{\Delta}=\ell$ and the result follows.
If $d<2\ell-2$ then $\gamma_{\ell-1}=0$ and the dimension is too small to be isomorphic to the two candidates, finishing this case.
Since we assume that $\Delta$ does not contain suspension vertices, by Lemma~\ref{lem:bound_dim}, $d<2\ell+1$. It remains to consider the case $d\in \{2\ell-2,2\ell-1,2\ell\}$.

Assume that $\pi_\Delta=2$, let $v_0\in\Delta$ be such that $\iota(v_0)=\pi_\Delta$, denote $L:=\lk_\Delta(v_0)$, the two antipodes of $v_0$ by $x$ and $y$ and the link of $\{x,y\}$ by $J$.
We have $\gamma_1^L=\ell-1$ and $\gamma_1^J\leq \ell-1$ and by Lemma~\ref{lem:iota_equals_2_eqn}, $\gamma_{\ell-1}^{\Delta}=\gamma_{\ell-1}^{L}+\gamma_{\ell-2}^J$.
By Theorem~\ref{thm:bound_index}, $\gamma_{\ell-1}^L\in\{0,1\}$.

\textbf{Case $\gamma_{\ell-1}^L=1$}. Then  $L\cong \Sigma^m\join^{\ell-1}C_5$, for some $m\geq 0$ and by Lemma~\ref{lem:equators_type}, $J$ is isomorphic to either $\Sigma^{m-1}\join^{\ell-1}C_5$ or $\Sigma^{m+1}\join^{\ell-2}C_5$.
Therefore, $\gamma_{\ell-1}^\Delta\in\{2,\ell\}$.
In the case $\gamma_{\ell-1}^\Delta=2$, it means that $J$ is isomorphic to $\Sigma^{m+1}\join^{\ell-2}C_5$.
Further, $\Delta$ is obtained by taking a suspension over $L$, by vertices $v_0$ and say $v_1$, followed by a vertex splitting of $v_1$ where the equator of the splitting is $J$.
By Lemma~\ref{lem:equators_type} this equator is the link of a vertex $v_2$ of an induced $C_5$ in $L$,
which in turn makes~$J$ the link of the edge $e'=v_1v_2$ in $\Sigma L$.
Hence, the vertex-split is the same as the edge-subdivision at $e'$, making $\Delta$ a sphere of the second type.

\textbf{Case $\gamma_{\ell-1}^L=0$}. Then $\gamma_{\ell-1}^{\Delta}=\gamma_{\ell-2}^{J}\leq \ell-1$ by the induction hypothesis on $J$.
By Theorem~\ref{thm:bound_index} for $\Delta$, $\gamma_{\ell}^{\Delta}=0$.
We now show that $\gamma_{\ell-1}^{\Delta}\in\{0,1,2\}$ in this case.

\textbf{Subcase $\gamma_1^J=\ell-1$}.
Also $\gamma_1^L=\ell-1$, and $J$ is an equator in $L$, thus $L=\Sigma J$, and by Theorem~\ref{thm:polar_values} $\Delta=J \join C_5$.
As $\gamma_{\ell-1}^L=0$ also $\gamma_{\ell-1}^J=0$, and by the induction hypothesis $\gamma_{\ell-2}^{J}\in\{0,1,2\}$, and equals $2$ if and only if $J$ is of the first or second type.
Consequently, $\gamma_{\ell-1}^{\Delta}\in\{0,1,2\}$, and equals $2$ if and only if $\Delta$ is of the first or second type.

\textbf{Subcase $\gamma_1^J<\ell-1$}.
By Theorem~\ref{thm:bound_index}, $0\leq \gamma_{\ell-2}^{J}\leq 1$, and thus $0\leq \gamma_{\ell-1}^{\Delta}\leq 1$.

Finally assume that $\pi_\Delta\geq 3$.
Since $d\geq 3$, Corollary~\ref{cor:polar_size_3_degree_gamma} states that $\gamma_{j}=0$, for all $j\geq \ell-1 \geq \ell-\pi_\Delta+2$.
The complex $\Delta$ can not be either types since $\gamma^{\Upsilon_1}(t)$ and $\gamma^{\Upsilon_2}(t)$ have degree $\ell-1$.
\end{proof}

A closer look on the proof of Theorem~\ref{thm:gamma_ell-1} gives the following technical result, which restrict the possibilities of having $\gamma_{\ell-1}=1$ in Theorem~\ref{thm:gamma_ell-1}.
It
will be used to give infinitely many counterexamples to the converse of Nevo--Petersen's Conjecture~\ref{conj:Nevo-Petersen}.

\begin{theorem}\label{thm:poly}
Let $k\geq 3$, $q_k(t)=(1+t)^k$,
and $r(t)$ be a polynomial in $\mathbb{Z}[t]$ of degree at most $k-2$ with constant term $1$ which is not divisible by $(1+t)$.
The polynomial $p(t)=q_k(t)+tr(t)$ is not the $\gamma$-polynomial of any flag homology sphere.
\end{theorem}

\begin{proof}
Assume to the contrary, that it is the case.
Let $\Delta$ be a flag homology sphere such that $\gamma^\Delta(t)=p(t)=\sum_{i\geq 0}p_it^i$ of smallest dimension $d-1$.
As $\gamma_1^\Delta=\binom{k}{1}+1=k+1$ and $\gamma_k^{\Delta}=\binom{k}{k}+0=1$, the dimension is bounded by $2k\leq d<2k+3$ by Lemma~\ref{lem:bound_dim}.
Let $v_0$ be a vertex of $\Delta$ with $\iota(v_0)=\pi_\Delta$ and denote $L:=\lk_\Delta(v_0)$.
By the minimality of $d$, $\iota(v_0)\geq 2$.
The dimension of $L$ is $d-2$ and $\gamma_1^L=k-\iota(v_0)+2\leq k$.

If $\pi_\Delta=2$, as usual denote the antipodes of $v_0$ by $x$ and $y$, let $e:=\{x,y\}\in \Delta$ and $J:=\lk_\Delta(e)$, and by Lemma~\ref{lem:iota_equals_2_eqn} conclude
\[
 \gamma^\Delta(t)=\gamma^{L}(t)+t\gamma^J(t).
\]
By Lemma~\ref{lem:polar_values_i}, $\gamma_1^{L}=k$.
By Theorem~\ref{thm:bound_index}, $\gamma_j^{L}=0$ for all $j>k$.
Since $p_j=0$ for all $j>k$, therefore $\gamma_j^{J}=0$ for all $j>k-1$.

Knowing that $p_k=\gamma_k^\Delta=1=\gamma_k^L +\gamma_{k-1}^J$ and that $\gamma_1^L=k$, by Theorem~\ref{thm:bound_index}, there are two possibilities: either $\gamma_k^L=1$ and $\gamma_{k-1}^J=0$, or $\gamma_k^L=0$ and $\gamma_{k-1}^J=1$.

{\bf Case $\gamma_k^L=1$.} By Theorem~\ref{thm:bound_index}, $L\cong\Sigma^m\join^k C_5$, for a certain $m\geq 0$.
Hence $\gamma^L(t)=q_k(t)$ and $\gamma^J(t)=r(t)$.
Since $J$ is the link of an edge of $\Delta$, it is an equator of $L$.
By Lemma~\ref{lem:equators_type}, $J$ is either $\Sigma^{m-1}\join^{k}C_5$ or $\Sigma^{m+1}\join^{k-1}C_5$.
Hence $(1+t)^2$ divides $\gamma^J(t)$, as $k\geq3$.
Since $(1+t)$ does not divide $\gamma^J(t)=r(t)$ this is not possible.

{\bf Case $\gamma_k^{L}=0$ and $\gamma_{k-1}^{J}=1$.} By Theorem~\ref{thm:bound_index} and as $\gamma_1^J\leq\gamma_1^{L}=k$, it follows that $\gamma_1^J\in\{k-1,k\}$.

{\bf Subcase $\gamma_1^J=k-1$.} Then $J\cong \Sigma^m\join^{k-1} C_5$, for a certain $m\geq 0$.
Using the values for $\gamma_1^J$ and~$\gamma_1^L$, there are exactly $3$ vertices in $\del_L(J)$, forming an edge and isolated vertex.
Denote by $w_0$ the isolated vertex of $\verti(L)\setminus \verti(J)$, so $\iota_L(w_0)=2$.
We have $\gamma^L(t)=\gamma^J(t)+t\gamma^M(t)$, where $M$ is the link of the edge of antipodes of $w_0$ in $L$, by Lemma~\ref{lem:iota_equals_2_eqn}.
Therefore
\[
 \gamma^\Delta(t)=q_k(t)+tr(t) = \gamma^L(t)+t\gamma^J(t)=(1+t)\gamma^J(t)+t\gamma^M(t)=(1+t)^k+t\gamma^M(t).
\]
Hence $r(t)=\gamma^M(t)$.
But $M$ is an equator of $J\cong\Sigma^m\join^{k-1} C_5$, hence $M\cong \Sigma^{m-1}\join^{k-1}C_5$ or $M\cong \Sigma^{m+1}\join^{k-2}C_5$.
Since $k\geq 3$, $(1+t)$ divides $r(t)$ and this contradicts the assumption.

{\bf Subcase $\gamma_1^J=k$.} Then $L=\Sigma J$ and by Theorem~\ref{thm:extract_cycle}, $\Delta\cong J \join C_5$.
Hence $p(t) = \gamma^\Delta(t) = (1+t)\gamma^J(t) = q_k(t) +tr(t) = (1+t)^k + tr(t)$.
Equivalently, $(1+t)\left(\gamma^J(t)-(1+t)^{k-1}\right)=tr(t)$, and $(1+t)$ has to divide $r(t)$, which is impossible by assumption.

Finally, assume $\pi_{\Delta}\geq 3$.
Since $k\geq 3$, the dimension of $\Delta$ satisfies $d\geq6$.
Because of the latter bound on the dimension, we can apply Corollary~\ref{cor:polar_size_3_degree_gamma}, so that $\pi_\Delta-1\geq 2$ divides the highest degree coefficient of $\gamma^{\Delta}(t)$, i.e., $\gamma_k^\Delta=1$, which is a contradiction.
\end{proof}

\begin{corollary}
Let $k\geq 3$ and $\Lambda$ be the simplicial complex consisting of a $(k-1)$-simplex and a disconnected vertex.
The $f$-vector of $\Lambda$ is not the $\gamma$-vector of any flag homology sphere, despite being the $f$-vector of a balanced flag complex $\Lambda$.
\end{corollary}

\begin{proof}
The $f$-polynomial of $\Lambda$ is $f(t)=q_k(t)+t$, so that $r(t)=1$ in Theorem~\ref{thm:poly}.
\end{proof}

\bibliographystyle{alpha}
\bibliography{bibliography}

\end{document}